\documentclass[11pt]{amsart}
\usepackage[colorlinks=true,citecolor=magenta,linkcolor=magenta]{hyperref}
\usepackage{amssymb,amscd,amsmath,graphicx}
\usepackage{enumerate}
\usepackage{fullpage}

\numberwithin{equation}{section}
\newtheorem{theorem}{Theorem}[section]
\newtheorem{lemma}[theorem]{Lemma}
\newtheorem{proposition}[theorem]{Proposition}
\newtheorem{corollary}[theorem]{Corollary}

\theoremstyle{definition}
\newtheorem{definition}[theorem]{Definition}

\newtheorem{example}[theorem]{Example}
\newtheorem*{acknowledgment}{Acknowledgments}

\DeclareMathOperator{\Tor}{Tor}
\DeclareMathOperator{\reg}{reg}
\DeclareMathOperator{\depth}{depth}

\DeclareMathOperator{\Ass}{Ass}

\DeclareMathOperator{\Min}{Min}

\DeclareMathOperator{\chara}{char}

\DeclareMathOperator{\grade}{grade}

\newcommand{\ZZ}{{\mathbb Z}}
\newcommand{\NN}{{\mathbb N}}

\newcommand{\up}[1]{^{\tt #1}\!}

\def\kk{{\Bbbk}}

\def\pp{{\mathfrak p}}
\def\qq{{\mathfrak q}}

\begin{document}
	
	\title{Binomial expansion for saturated and \\ symbolic powers of sums of ideals}
	
	\author{Huy T\`ai H\`a}
	\address{Tulane University \\ Department of Mathematics \\
		6823 St. Charles Ave. \\ New Orleans, LA 70118, USA}
	\email{tha@tulane.edu}
	%\urladdr{???}
	
	\author{A.V. Jayanthan}
	\address{Department of Mathematics, I.I.T. Madras, Chennai - 600036, INDIA}
	\email{jayanav@iitm.ac.in}
	
	\author{Arvind Kumar}
	\address{Department of Mathematics, Chennai Mathematical Institute, Siruseri
		Kelambakkam, Chennai, INDIA - 603103}
	\email{arvindkumar@cmi.ac.in}
	
	\author{Hop D. Nguyen}
	\address{Institute of Mathematics, VAST, 18 Hoang Quoc Viet, Cau Giay, 10307 Hanoi, VIETNAM}
	\email{ngdhop@gmail.com}
	
	
	\begin{abstract}
There are two different notions for symbolic powers of ideals existing in the literature, one defined in terms of associated primes, the other in terms of minimal primes. Elaborating on an idea known to Eisenbud, Herzog, Hibi, and Trung, we interpret both notions of symbolic powers as suitable saturations of the ordinary powers. We prove a binomial expansion formula for saturated powers of sums of ideals. This gives a uniform treatment to an array of existing and new results on both notions of symbolic powers of such sums: binomial expansion formulas, computations of depth and regularity, and criteria for the equality of ordinary and symbolic powers.
	\end{abstract}
	
\maketitle
	
%%%%%%%%%%%%%%%%%%%%%%%%%%%%%%%%%%
%%%%%%%%%%%%%%%%%%%%%%%%%%%%%%%
	
\section{Introduction} \label{sec.intro}

Let $\kk$ be a field, let $A$ and $B$ be Noetherian $\kk$-algebras such that $R = A \otimes_\kk B$ is also Noetherian. Let $I \subseteq A$ and $J \subseteq B$ be ideals, and let $I+J\subseteq R$ denote the ideal $IR+JR$. The following binomial expansion for the symbolic powers of $(I+J)$ was given in \cite{HNTT}:
\begin{align}
	\label{eq.bin}
	(I+J)^{(s)} = \sum_{i=0}^s I^{(i)} J^{(s-i)}.
\end{align}	
This binomial expansion has been well received and seen many applications (cf. \cite{BHJT, EH2020, ERT2020, KKS2021, LS2021, MV2021, OR2019, SF2020, SW2019, W2018}).

In working with symbolic powers of ideals, it is often remarked that there are two different notions for symbolic powers existing in the literature, both being much investigated. This has led to many difficulties and sometimes confusions in using formula (\ref{eq.bin}). For instance, incorrect applications of (\ref{eq.bin}) were used in \cite{BHJT, HNTT, JKM2021} to study resurgence numbers, asymptotic resurgence numbers, and Waldschmidt constants. The aim of this paper is to prove that the binomial expansion (\ref{eq.bin}) provided in \cite{HNTT} is in fact valid for both known definitions of symbolic powers, allowing more applications and filling the gaps in \cite{BHJT, HNTT, JKM2021}. Particularly, our work shows that the consequences of formula (\ref{eq.bin}) in\cite[Lemmas 2.5, 3.3, and Theorem 2.7]{BHJT}, \cite[Corollary 3.8]{HNTT} and \cite[Theorems 3.6 and 3.9]{JKM2021} remain true as stated.

For an ideal $I$ in an arbitrary commutative ring $A$  and a positive integer $s$, the two existing notions of the \emph{$s$-th symbolic power} of $I$ are given by
$$\up{m}I^{(s)} = \bigcap_{\pp \in \Min_A(I)} (I^sA_\pp \cap A) \text{ and } \up{a}I^{(s)} = \bigcap_{\pp \in \Ass_A(I)} (I^sA_\pp \cap A).$$
The difference between these notions lies in the set of associated primes over which the intersection takes place. If $I$ has no embedded associated primes, i.e., $\Ass_A(I)=\Min_A(I)$, then the two notions are the same. This is the case for ideals of geometrical and combinatorial interest, including radical ideals (squarefree monomial ideals and prime ideals in particular). In general, for ideals with embedded associated primes, the two symbolic powers are different (Example \ref{ex.symbMsat}). In existing literature, there are no customarily distinctive notations for these notions of symbolic powers. In this paper, to distinguish the two, we shall use superscripts ${\tt m}$ and ${\tt a}$ as above. The binomial expansion given in \cite{HNTT} was for $\up{m}I^{(s)}$, while for applications in \cite{BHJT, HNTT, JKM2021} it was $\up{a}I^{(s)}$ that was examined.

It turns out that both of these notions of symbolic powers fall under a more general umbrella, namely, they both can be realized as the saturated powers of $I$ with respect to appropriate ideals; see Lemmas   \ref{lem.symbSat_min} and \ref{lem.symbSat_ass}. We shall address this much more general notion of saturated powers.

\begin{definition}
	\label{def.satpower}
	Let $I, K$ be ideals in a commutative ring $A$, and let $s$ be a positive integer. The \emph{$s$-th saturated power} of $I$ \emph{with respect to $K$} is defined to be
	$$I^{(s)}_K := I^s : K^\infty = \bigcup_{t \ge 1} (I^s : K^t).$$
\end{definition}
The fact that symbolic powers can be realized as saturation was known to several authors, for example Herzog-Hibi-Trung \cite[Section 3]{HHT} and Eisenbud \cite[Proposition 3.13]{Eis} (if only in a disguised form). We illustrate the fertility of this idea in the present paper, and derive various results for both notions of symbolic powers under a common framework. Our main result establishes a binomial expansion of the form (\ref{eq.bin}) for saturated powers of $(I+J)$, where again $A, B$ are Noetherian $\kk$-algebras such that $A\otimes_\kk B$ is also Noetherian, and $I \subseteq A$ and $J \subseteq B$ are ideals. Specifically, we prove the following theorem.

\medskip

\noindent\textbf{Theorem \ref{binomial_expansion}.} Let $I,K \subseteq A$ and $J,L \subseteq B$ be ideals. Then, for any $s \in \NN$, we have
\begin{align} \label{eq.binsat}
	(I+J)^{(s)}_{KL}=\sum_{i=0}^s I^{(i)}_K J^{(s-i)}_L .
\end{align}

A direct application of Theorem \ref{binomial_expansion} particularly shows that the binomial expansion for symbolic powers of $(I+J)$ in (\ref{eq.bin}) is valid for both the known notions, using minimal and associated primes, of symbolic powers; see Theorem \ref{thm.symbBIN}. Theorem \ref{binomial_expansion} also allows us to derive formulas to compute the depth and the regularity of saturated powers of $I+J$ in terms of those of $I$ and $J$; see Theorem \ref{thm_depth_reg_saturatedpowers}. These formulas generalize those given in \cite{HNTT} and, at the same time, exhibit that the formulas in \cite{HNTT} are valid for symbolic powers defined by all associated primes as well; see Corollary \ref{cor_depth_reg_ass-symbpower}. Another application of Theorem \ref{binomial_expansion} is a criterion for the equality of ordinary and ``associated'' symbolic powers of $I+J$ (Corollary \ref{cor.equality}). There is an analogous result for ``minimal'' symbolic powers (\cite[Corollary 3.5]{HNTT}) whose proof cannot be adapted to handle ``associated'' symbolic powers.

Let $Q_s$ denote the right hand side of (\ref{eq.binsat}). To prove Theorem \ref{binomial_expansion}, we establish both containments $(I+J)^{(s)}_{KL} \supseteq Q_s$ and $(I+J)^{(s)}_{KL} \subseteq Q_s$.  The first containment, namely, $(I+J)^{(s)}_{KL} \supseteq Q_s$, is achieved by directly showing that $I^{(i)}_K J^{(s-i)}_L \subseteq (I+J)^{(s)}_{KL}$ for any $0\le i\le s$; see Proposition \ref{prop_sub}. The second containment, namely, $(I+J)^{(s)}_{KL} \subseteq Q_s$, on the other hand, is more involved and does not follow from similar lines of arguments as those given in \cite{HNTT}. To derive at this later containment, we examine colon ideals of the form $(I+J)^s : (KL)^t$ and make use of the binomial expansion for the usual power $(I+J)^s$.

The paper is outlined as follows. In the next section, we collect some basic facts about saturated powers and prove that, with respect to appropriate ideals, symbolic powers are saturated powers. These are done in Lemmas \ref{lem.symbSat_min} and \ref{lem.symbSat_ass}. We also recall a few important results about associated primes of powers of sums of ideals that will be used later. In Section \ref{sec.bin}, we prove our main result, Theorem \ref{binomial_expansion}, establishing a binomial expansion for a general saturated power of $(I+J)$. This yields a criterion for the equality of ordinary and certain saturated powers of $I+J$; see Theorem \ref{thm.equalityConverse}. In Section \ref{sec.applications}, we specify Theorem \ref{binomial_expansion} to symbolic powers and show that the binomial expansion in (\ref{eq.bin}) holds true for both definitions of symbolic powers. This is done in Theorem \ref{thm.symbBIN}. We also apply Theorem \ref{binomial_expansion} to get formulas for the depth and regularity of saturated powers of $(I+J)$. This is done in Theorem \ref{thm_depth_reg_saturatedpowers}.

\begin{acknowledgment} The first named author acknowledges supports from Louisiana Board of Regents and the Simons Foundation. The third named author thanks the Infosys foundation, and Sciences and Engineering Research Board (PDF/2020/001436), India for the financial support. The last named author thanks the Vietnam Institute for Advanced Study in Mathematics (VIASM) for its hospitality and generous support. Many of our examples are computed using Macaulay 2 \cite{GS} package.
\end{acknowledgment}	
%%%%%%%%%%%%%%%%%%%%%%%%%%%%%%%

\section{Saturated and symbolic powers of ideals} \label{sec.powers}
For standard notations and terminology of commutative algebra, we refer the interested reader to \cite{BH, Eis}. The recent survey \cite{DDS+} discusses various topics on symbolic powers defined in terms of all associated primes.

In this section, we collect basic properties of saturated powers of an ideal and show that its symbolic powers are also saturated powers with respect to appropriate ideals. We also recall a few results on associated primes of powers of sums of ideals that will be used in later sections. Throughout this section, $A$ will denote a Noetherian commutative ring.

We start by recalling from Definition \ref{def.satpower} that for ideals $I, K \subseteq A$ and a positive integer $s$, the $s$-th saturated power of $I$ with respect to $K$ is given by
$$I^{(s)}_K = I^s : K^\infty.$$
In other words, if $I^s = \bigcap_{\pp \in \Ass_A(I^s)} Q(\pp)$ is an irredundant primary decomposition of $I^s$, where $Q(\pp)$ is the $\pp$-primary component of $I^s$, then
\begin{equation}
\label{eq_satpower_andprimarydecomp}
I^{(s)}_K = \bigcap_{\pp \in \Ass_A(I^s), \  K \not\subseteq \pp} Q(\pp). 
\end{equation}

We proceed in showing that symbolic powers are indeed saturated powers, elaborating on an idea known to Herzog-Hibi-Trung \cite[Section 3]{HHT} and Eisenbud \cite[Proposition 3.13]{Eis} (in an indirect form). For an ideal $I \subseteq A$, set
$\Ass_A^*(I) := \bigcup_{n=1}^\infty \Ass_A(I^n),$ which is a finite set, by the proof of \cite[Theorem 2.11]{Rat1976}.
Recall that
$$
\up{m}I^{(s)} = \bigcap_{\pp \in \Min_A(I)}(I^s A_\pp \cap A).
$$
Note that $\bigcup_{\pp \in \Min_A(I)} \pp$ consists of zero-divisors of $A/\sqrt{I}$. By prime avoidance, it is easy to see that for $y \not\in \bigcup_{\pp \in \Min_A(I)} \pp$, we have
\begin{align} \label{eq.colon}
	I^s : y \subseteq \ \up{m}I^{(s)}.
\end{align}

In what follows we employ the usual convention that an empty intersection of ideals in $A$ is the ring $A$.

\begin{lemma}[``Minimal" symbolic powers as saturation, cf. {\cite[Section 3]{HHT}}]
	\label{lem.symbSat_min}
	Let $I \subseteq A$ be any ideal and let $s$ be a positive integer. Set
	$$K = \bigcap_{\pp \in \Ass_A^*(I) \setminus \Min_A(I)} \pp \text{ and } K_s = \bigcap_{\pp \in \Ass_A(I^s) \setminus \Min_A(I)} \pp.$$
	Let $x \in K$ be any element that is regular on $A/\sqrt{I}$ if $\Ass_A^*(I) \setminus \Min_A(I)\neq \emptyset$, and $x = 1$ otherwise. Then,
	$$\up{m}I^{(s)} = I^{(s)}_K = I^{(s)}_{K_s} = I^{(s)}_{(x)}.$$
	In particular, given an arbitrary irredundant primary decomposition of $I^s$,  $\up{m}I^{(s)}$ is the intersection of its components that are primary to elements of $\Min_A(I)$.
\end{lemma}

\begin{proof} 
	Observe that $(x) \subseteq K \subseteq K_s$. Thus,  $ I^{(s)}_{K_s} \subseteq I^{(s)}_K \subseteq I^{(s)}_{(x)}. $  To establish the desired equality, it suffices to show that
	$$I^{(s)}_{(x)} \subseteq \ \up{m}I^{(s)} \subseteq I^{(s)}_{K_s}.$$
	
Let $u \in I^{(s)}_{(x)}$ be any element. Then, for some positive integer $t$, we have that $u \in I^s : x^t.$ This, together with (\ref{eq.colon}), implies that $u \in \ \up{m}I^{(s)}$ as $x^t \notin \bigcup_{\pp \in \Min_A(I)} \pp .$ Therefore, $I^{(s)}_{(x)} \subseteq \ \up{m}I^{(s)}$.

Now, let $v \in \ \up{m}I^{(s)} \setminus I^s$ be any element, if exists.  By the definition of symbolic powers and the prime avoidance, we have $I^s : v \not \subseteq  \bigcup_{\pp \in \Min_A(I)} \pp$. This, since $\Ass_A(I^s:v) \subseteq \Ass_A(I^s)$, implies that
	\[
	\Ass_A(I^s: v) \subseteq \{\pp ~\big|~ \pp \in \Ass_A(I^s) \setminus \Min_A(I)\}.
	\]
Since $(I^s : v)$ is finitely generated, there exists a positive integer $q$ such that  $(I^s:v)$ contains  $\left(\bigcap_{\pp \in \Ass_A(I^s:v)} \pp\right)^q$. Thus,
	\[
	K_s^q =  \left(\mathop{\bigcap_{\pp \in \Ass_A(I^s) \setminus \Min_A(I)}} \pp\right)^q  \subseteq \left(\bigcap_{\pp \in \Ass_A(I^s:v)} \pp\right)^q \subseteq I^s: v.
	\]
	This implies that  $v \in I^s : K_s^{\infty}$. Hence, $\up{m}I^{(s)} \subseteq I^{(s)}_{K_s}$. The desired equality follow.
	
	Finally, the last assertion follows from the established equality $\up{m}I^{(s)} = I^{(s)}_{K_s}$ and Equation \eqref{eq_satpower_andprimarydecomp}.
\end{proof}

Recall also that
$\up{a}I^{(s)} = \bigcap_{\pp \in \Ass_A(I)} (I^s A_\pp \cap A).$
Note that $\bigcup_{\pp \in \Ass_A(I)} \pp$ consists of zero-divisors of $A/I$. Similar to the observation that led to (\ref{eq.colon}), it can be seen that with this notion of symbolic powers, for any $y \not\in \bigcup_{\pp \in \Ass_A(I)} \pp$, we have
\begin{align}
	\label{eq.colonAss}
	I^s : y \subseteq \ \up{a}I^{(s)}.
\end{align}

\begin{lemma}[``Associated" symbolic powers as saturation]
	\label{lem.symbSat_ass}
	Let $I \subseteq A$ be any ideal and let $s$ be a positive integer. Set
	$$K = \bigcap_{\substack{\pp \in \Ass_A^*(I) \\ \grade(\pp, A/I) \ge 1}} \pp \text{ and } K_s = \bigcap_{\substack{\pp \in \Ass_A(I^s) \\ \grade(\pp,A/I) \ge 1}} \pp.$$
	Let $x \in K$ be any element that is regular on $A/I$ if $\{\pp \in \Ass_A^*(I) \mid \grade(\pp,A/I)\ge 1\}\neq \emptyset$, and $x = 1$ otherwise. Then,
	$$\up{a}I^{(s)} = I^{(s)}_K = I^{(s)}_{K_s} = I^{(s)}_{(x)}.$$
	In particular, given an arbitrary irredundant primary decomposition of $I^s$,  $\up{a}I^{(s)}$ is the intersection of its components that are primary to ideals $\pp\in \Ass_A(I^s)$ such that $\grade(\pp,A/I)=0$, i.e., $\pp$ is contained in an element of $\Ass_A(I)$.
\end{lemma}

\begin{proof}
As in Lemma \ref{lem.symbSat_min}, it is easy to see that $I^{(s)}_{K_s} \subseteq I^{(s)}_K \subseteq I^{(s)}_{(x)}$, and it suffices to prove the following inclusions:
	$$I^{(s)}_{(x)} \subseteq \ \up{a}I^{(s)} \subseteq I^{(s)}_{K_s}.$$
	
The first inclusion follows by (\ref{eq.colonAss}) and exactly the same argument as in Lemma \ref{lem.symbSat_min}. The second inclusion is also established by a similar line of arguments as that in Lemma \ref{lem.symbSat_min}. Specifically, consider any element $v \in \ \up{a}I^{(s)} \setminus I^s$ if exists. Again, by the definition and prime avoidance, we have $I^s : v \not \subseteq \bigcup_{\pp \in \Ass_A(I)} \pp.$ This, since $\Ass_A(I^s :v) \subseteq \Ass_A(I^s)$, implies that
$$\Ass_A(I^s :v) \subseteq \{\pp ~\big|~ \pp \in \Ass_A(I^s) \text{ and } \grade(\pp, A/I) \ge 1\}.$$
The argument continues exactly line by line as that in Lemma \ref{lem.symbSat_min}, replacing the phrase ``$\pp \in \Ass_A(I^s)\setminus \Min_A(I)$'' by ``$\pp\in \Ass_A(I^s) \text{ and } \grade(\pp, A/I) \ge 1$''.

Finally, the last assertion follows from the established equality $\up{a}I^{(s)} = I^{(s)}_{K_s}$ and Equation \eqref{eq_satpower_andprimarydecomp}.
\end{proof}

\begin{example}
\label{ex.symbMsat}
Let $I = (a^2, ab) = (a)(a,b) \subseteq A = \kk[a,b]$. Then, for $s \in \NN$, $I^s = (a^s)(a,b)^s.$ This implies that $\Ass_A^*(I) = \{(a), (a,b)\}$ and $\Min_A(I) = \{(a)\}$. Particularly, we have
$$K = K_s = (a,b),$$
and $x$ can be chosen to be $x = b$. Now,
$$I^{(s)}_K = I^{(s)}_{K_s} = I^{(s)}_{(b)} = \ \up{m}I^{(s)} = (a^s).$$

Since the set $\{\pp\in \Ass_A(I^s) \text{ and } \grade(\pp, A/I) \ge 1\}$ is empty, we see that $\up{a}I^{(s)}=I^s$ for all $s\ge 1$. Hence, the two known definitions give different $s$-th symbolic powers of $I$ for all $s \ge 1$.
\end{example}

\begin{example}
	\label{ex.symbAsat}
	Let $I = (u^5, u^4v, uv^4, v^5, u^3v^3, u^3v^2w+u^2v^3z) \subseteq A = \kk[u,v,w,z].$ Note that $I$ is primary to $(u,v)$. Direct computation shows that
	$$I^2 = (u^{10}, u^9v, u^8v^2, u^6v^4, u^5v^5, u^4v^6,u^2v^8,uv^9,v^{10},u^7v^3w, u^3v^7w, u^7v^3z, u^3v^7z),$$
	and $I^s = (u,v)^{5s}$ for $s \ge 3$. Thus,
	$\Ass_A^*(I) = \{(u,v),(u,v,w,z)\}.$
	Particularly, $K = (u,v,w,z)$ and we can chose $x = w$. We also have $K_2 = K$ and $K_s = (1)$ for $s \not= 2$. Therefore, for $s \ge 2$,
	$$I^{(s)}_K = I^{(s)}_{K_s} = I^{(s)}_{(x)} = \ \up{a}I^s = (u,v)^{5s}.$$
\end{example}

We end the section by paving the way to  the main result in Section \ref{sec.applications}, with Lemmas \ref{lem_satpower_min_symbpower} and \ref{lem_satpower_ass_symbpower}. For this, it is necessary to recall a few facts on associated primes of powers of sums, and the behavior of grade with respect to certain tensor products.

\begin{lemma}
	\label{lem_Ass_tensor}
	Let $M$ and $N$ be nonzero finitely generated modules over $A$ and $B$, respectively. Then,
	\begin{align*}
		\Ass_R(M \otimes_\kk N) & = \bigcup_{\substack{\pp \in \Ass_A(M) \\ \qq \in \Ass_B(N)}} \Min_R \left(\pp + \qq\right), \text{ and } \\
		\Min_R(M \otimes_\kk N) & = \bigcup_{\substack{\pp \in \Min_A(M) \\ \qq \in \Min_B(N)}} \Min_R \left(\pp + \qq\right).
	\end{align*}
	More precisely, $P \in \Ass_R(M \otimes_\kk N)$ (respectively, $\Min_R(M \otimes_\kk N)$) if and only if $\pp = P \cap A \in \Ass_A(M)$ (respectively, $\Min_A(M)$), $\qq = P \cap B \in \Ass_B(N)$ (respectively, $\Min_B(N)$), and $P \in \Min_R\left(\pp+\qq\right).$
\end{lemma}

\begin{proof}
	The assertion follows from \cite[Theorem 2.5]{HNTT} and its proof.
\end{proof}

\begin{lemma}
\label{lem_ass_powers}
Let $I$ and $J$ be nonzero proper ideals of $A$ and $B$, respectively. Then for any $s \in \NN$, we have
 \begin{align*}
 \bigcup_{i=1}^s \mathop{\bigcup_{\pp \in \Ass_A(I^{i-1}/I^i)}}_{\qq \in \Ass_B(J^{s-i}/J^{s-i+1})} \Min_R(\pp+\qq)  &\subseteq  \Ass_R ((I+J)^s), \text{ and }\\
\Ass_R ((I+J)^s)  &\subseteq  \bigcup_{i=1}^s \mathop{\bigcup_{\pp \in \Ass_A(I^i)}}_{\qq \in \Ass_B(J^{s-i}/J^{s-i+1})} \Min_R(\pp+\qq).
 	\end{align*}
More precisely, if $P \in \Ass_R((I+J)^s)$ and $\pp= P \cap A, \qq = P \cap B$, then there exists $1\le i\le  s$ such that  $\pp \in \Ass_A(I^i)$, $\qq \in \Ass_B(J^{s-i}/J^{s-i+1})$, and $P \in \Min_R(\pp+\qq)$.
\end{lemma}

\begin{proof} The two containment are contained in \cite[Theorem 4.1]{NT}. Moreover, the proof of \cite[Theorem 4.1]{NT} shows that
	\[
	\Ass_R ((I+J)^s)  \subseteq \bigcup_{i=1}^s \Ass_R \left((A/I^i) \otimes_\kk (J^{s-i}/J^{s-i+1})\right).
	\]
Hence, the remaining assertion follows by using Lemma \ref{lem_Ass_tensor}.
\end{proof}

\begin{lemma}
\label{lem_grade}
Let $\pp \subset A, \qq \subset B$ be prime ideals. Let $M, N$ be finitely generated modules over $A,B$, respectively. Let $P\in \Ass_R(\pp +\qq)$ be a prime ideal of $R$. Then there is an equality
\[
\grade(P,M\otimes_\kk N)=\grade(\pp ,M)+\grade(\qq,N).
\]
\end{lemma}
\begin{proof}
By Lemma \ref{lem_Ass_tensor}, $\pp=P\cap A, \qq=P\cap B$. It is clear that $(M \otimes_\kk N)_P = M_\pp \otimes_{A_\pp} (A \otimes_\kk N)_P.$
Since the map $A \to R$ is flat, so is the map $A_\pp \to R_P$.
Applying \cite[Chap. IV, (6.3.1)]{Gr}, we have
$$\depth (M\otimes_\kk N)_P   =  \depth M_\pp + \depth k(\pp) \otimes_{A_\pp} (A \otimes_\kk N)_P,$$
where $k(\pp)$ denotes the residue field of $A_\pp$.
Since $k(\pp) = (A/\pp)_\pp$, we have
$((A/\pp) \otimes_\kk N)_P =  k(\pp) \otimes_{A_\pp} (A \otimes_\kk N)_P$. Therefore,
$$\depth (M\otimes_\kk N)_P   =  \depth M_\pp + \depth ((A/\pp) \otimes_\kk N)_P.$$
Since the map $B \to R$ is flat, by similar arguments
$$\depth ((A/\pp) \otimes_\kk N)_P = \depth N_\qq + \depth ((A/\pp) \otimes_\kk (B/\qq))_P.$$

Note that $(A/\pp) \otimes_\kk (B/\qq) = R/(\pp +\qq)$. Hence $ \depth ((A/\pp) \otimes_\kk (B/\qq))_P=\depth (R/(\pp+\qq))_P=0$, as $P\in \Ass_R(\pp+\qq)$. From the above equalities we get
$\depth (M\otimes_\kk N)_P = \depth M_\pp +\depth N_\qq,$ namely
$
\grade(P,M\otimes_\kk N)=\grade(\pp ,M)+\grade(\qq,N).
$
The proof is concluded.
\end{proof}

The next two results are useful to the proof of the main result in Section \ref{sec.applications}.
\begin{lemma}
\label{lem_satpower_min_symbpower}
Let $I$ and $J$ be nonzero proper ideals of $A$ and $B$, respectively. Set 
$$K = \bigcap_{\pp \in \Ass_A^*(I) \setminus \Min_A(I)} \pp  \text{ and } L = \bigcap_{\qq \in \Ass_B^*(J) \setminus \Min_B(J)} \qq.$$ 
Let $s\in \NN$ and $P\in \Ass_R((I+J)^s)$. Then $P\not\in \Min_R(I+J)$ if and only if $KL\subseteq P$. In particular, 
\[
(I+J)^{(s)}_{KL}=\ \up{m}(I+J)^{(s)}.
\]
\end{lemma}
\begin{proof}
 Since $P\in \Ass_R((I+J)^s)$, Lemma \ref{lem_ass_powers} implies the existence of $1\le i\le s$, $\pp \in \Ass_A(I^i)$, and $\qq \in \Ass_B(J^{s-i}/J^{s-i+1}) \subseteq \Ass_B(J^{s-i+1}) $ such that $P \in \Min_R(\pp+\qq)$. By Lemma \ref{lem_Ass_tensor}, $\pp=P\cap A, \qq=P\cap B$. Moreover, $P\not\in \Min_R(I+J)$ if and only if either $\pp\notin \Min_A(I)$ or $\qq \notin \Min_B(J)$. Hence, if $P\not\in \Min_R(I+J)$, then either $K\subseteq \pp$ or $L\subseteq \qq$, and in both cases, $KL\subseteq \pp+\qq \subseteq P$.
 
 Conversely, if $KL\subseteq P$, then without loss of generality, we may assume $K\subseteq P$. This implies $K\subseteq P\cap A=\pp$. Hence $K\neq (1)$ and $\pp$ contains an element of  $\Ass_A^*(I) \setminus \Min_A(I)$, so $\pp \notin \Min_A(I)$. This yields $P\not\in \Min_R(I+J)$.
 
For the remaining assertion, fix an irredundant primary decomposition of $(I+J)^s$. By Lemma \ref{lem.symbSat_min}, $\ \up{m}(I+J)^{(s)}$ is the intersection of components that are primary to ideals  $P \in \Min_R(I+J)$, i.e., by the first assertion, those $P\in \Ass_R((I+J)^s)$ such that $KL \not \subseteq P$. This, together with Equality \eqref{eq_satpower_andprimarydecomp}, gives the desired equality.
\end{proof}

The analog of Lemma \ref{lem_satpower_min_symbpower} for symbolic powers defined in terms of associated primes is given in the next lemma.

\begin{lemma}
\label{lem_satpower_ass_symbpower}
Let $I$ and $J$ be nonzero proper ideals of $A$ and $B$, respectively. Set 
$$K = \bigcap_{\substack{\pp \in \Ass_A^*(I) \\ \grade(\pp,A/I) \ge 1}} \pp \text{ and } L = \bigcap_{\substack{\qq \in \Ass_B^*(J) \\ \grade(\qq,B/J) \ge 1}} \qq.$$ 
Let $s\in \NN$ and $P\in \Ass_R((I+J)^s)$. Then $\grade(P,R/(I+J))\ge 1$ if and only if $KL\subseteq P$. In particular,
\[
(I+J)^{(s)}_{KL}=\ \up{a}(I+J)^{(s)}.
\]
\end{lemma}
\begin{proof}
 Since $P\in \Ass_R((I+J)^s)$, Lemma \ref{lem_ass_powers} implies the existence of $1\le i\le s$ such that $\pp \in \Ass_A(I^i)$, $\qq \in \Ass_B(J^{s-i}/J^{s-i+1}) \subseteq \Ass_B(J^{s-i+1}) $, and $P \in \Min_R(\pp+\qq)$. By Lemma \ref{lem_Ass_tensor}, $\pp=P\cap A, \qq=P\cap B$. By Lemma \ref{lem_grade},
 \[
\grade(P,R/(I+J))=\grade(\pp,A/I)+\grade(\qq,B/J).
 \]
 Hence if $\grade(P,R/(I+J))\ge 1$, then either $\grade(\pp,A/I)\ge 1$ or $\grade(\qq,B/J)\ge 1$. Thus either $K\subseteq \pp$ or $L\subseteq \qq$, and in both cases, $KL\subseteq \pp+\qq \subseteq P$.
 
 Conversely, if $KL\subseteq P$, then without loss of generality, we may assume $K\subseteq P$. This implies $K\subseteq P\cap A=\pp$. Thus $K\neq (1)$, and $\pp$ contains a regular element on $A/I$, namely $\grade(\pp,A/I)\ge 1$. This yields $\grade(P,R/(I+J))\ge 1$.
 
For the remaining assertion, fix an irredundant primary decomposition of $(I+J)^s$. By Lemma \ref{lem.symbSat_ass}, $\ \up{a}(I+J)^{(s)}$ is the intersection of the components that are primary to ideals  $P \in \Ass_R((I+J)^s)$ such that  $\grade(P,R/(I+J))=0$, i.e. by the first assertion, those $P\in \Ass_R((I+J)^s)$ such that $KL \not \subseteq P$. This, together with Equality \eqref{eq_satpower_andprimarydecomp}, gives the desired equality.
\end{proof}

%%%%%%%%%%%%%%%%%%%%%%%%%%%%%%%%

\section{Binomial expansion for saturated powers} \label{sec.bin}

In this section, we prove the main result of this paper which establishes a general binomial expansion for saturated powers of sums of ideals. Throughout this section, let $A$ and $B$ be Noetherian $\kk$-algebras such that $R = A \otimes_\kk B$ is also Noetherian.
For ideals $I \subseteq A$ and $J \subseteq B$, when the context is clear we shall also use $I$ and $J$ to represent their extensions to $R$. 

We begin by establishing one inclusion for the desired binomial expansion.

\begin{proposition}\label{prop_sub}
	Let $I$, $K$ be ideals of $A$, and let $J$, $L$ be ideals of $B$. Then,
	$$ \displaystyle \sum_{i=0}^s I^{(i)}_K J^{(s-i)}_L \subseteq (I+J)^{(s)}_{KL}.$$
\end{proposition}

\begin{proof}
	Note that $I \subseteq I+J$ and $KL \subseteq K$. Therefore, for each $i \ge 0$,
	$$(I^i:_A K^{\infty})R=I^i:_R K^{\infty} \subseteq (I+J)^i:_R (KL)^{\infty}.$$
	Similarly, for each $j \ge 0$, $(J^j :_BL^{\infty})R \subseteq (I+J)^j:_R (KL)^{\infty}.$ Consequently,  for $i, j \ge 0$,
	$$I^{(i)}_KJ^{(j)}_L \subseteq  \left((I+J)^i:_R (KL)^{\infty}\right) \left((I+J)^j:_R (KL)^{\infty}\right) \subseteq (I+J)^{i+j}:_R (KL)^{\infty}.$$
	Hence, the assertion follows.
\end{proof}

Before proving the reverse inclusion and thus establishing the desired binomial expansion, we shall present a few useful lemmas.

\begin{lemma}
	\label{lem_intersection}
	Let $I \subseteq A$ and $J \subseteq B$ be ideals. As ideals in $R$, we have $IJ = I \cap J.$
\end{lemma}

\begin{proof} The assertion is a known fact; see, for instance, \cite[Lemma 3.1]{HNTT} (also see \cite[Lemma 1.1]{HoaT} for ideals in polynomial rings).
\end{proof}

\begin{lemma}
	\label{lem_intersection_sum}
	Let $I, K\subseteq A$ and $J\subseteq B$ be ideals. The following equality holds:
	\[
	(I+J)\cap (K+J)=(I\cap K)+J.
	\]
\end{lemma}

\begin{proof}
	It suffices to prove that the left hand side is contained in the right hand side. Take $x\in (I+J)\cap (K+J)$. Write $x=x_1+y_1=x_2+y_2$, where $x_1\in I, x_2\in K, y_1,y_2\in J$. Then, by Lemma \ref{lem_intersection}, we have
	\[
	x_1-x_2=y_2-y_1 \in (I+K)\cap J=(I+K)J.
	\]
	Hence $x_1-x_2=u-v, u\in IJ, v\in KJ$. Thus $x_1-u=x_2-v \in I\cap K$, as $x_1,u\in I, x_2,v\in K$. Finally $x=x_1+y_1=(x_1-u)+(u+y_1) \in I\cap K+IJ+J=(I\cap K)+J$, as desired.
\end{proof}

\begin{lemma}
	\label{lem_intersection_binomsum_1}
	Let $\{I_i\}_{i\ge 0}$ and $\{J_j\}_{j \ge 0}$ be filtrations of ideals in $A$ and $B$, respectively. Then for all $s\ge 1$, the following equality holds:
	\[
	J_1 ~\bigcap ~ \left(\sum_{i=0}^{s-2}I_iJ_{s-i} + I_{s-1}\right)=\sum_{i=0}^{s-1}I_iJ_{s-i}.
	\]
\end{lemma}

\begin{proof}
	Since $\{J_j\}_{j \ge 0}$ is a filtration, the right hand side is contained in the left hand side. For the reverse containment, consider any element $x$ in the left hand side. Then, $x\in J_1$ and $x=x_1+y_1$, where $x_1\in \sum_{i=0}^{s-2}I_iJ_{s-i} \subseteq J_1$ and $y_1\in I_{s-1}$. Therefore, thanks to Lemma \ref{lem_intersection}, we have
	\[
	x-x_1=y_1 \in J_1\cap I_{s-1}=J_1I_{s-1}.
	\]
	Hence, $x=x_1+y_1\in \sum_{i=0}^{s-2}I_iJ_{s-i}+J_1I_{s-1}=\sum_{i=0}^{s-1}I_iJ_{s-i},$ as desired.
\end{proof}

\begin{lemma}
	\label{lem_intersection_binomsum_long}
	Let $\{I_i\}_{i \ge 0}$ and $\{K_i\}_{i \ge 0}$ be filtrations of ideals in $A$, and let $\{J_j\}_{j \ge 0}$ be a filtration of ideals in $B$. Assume, moreover, that $I_0=K_0=A$ and $J_0=B$. Then, for all $s\ge 1$, the following equality holds:
	\[
	\left(\sum_{i=0}^s I_iJ_{s-i}\right) ~ \bigcap ~ \left(\sum_{i=0}^s K_iJ_{s-i}\right) =\sum_{i=0}^s (I_i\cap K_i) J_{s-i}.
	\]
\end{lemma}

\begin{proof}
	We shall use induction on $s\ge 1$. When $s=1$, the desired statement is
	\[
	(I_1+J_1)\cap (K_1+J_1)=(I_1\cap K_1)+J_1,
	\]
	which is true by Lemma \ref{lem_intersection_sum}. Assume that $s\ge 2$. Clearly,
	\begin{equation}
		\label{eq_intersection}
		\left(\sum_{i=0}^s I_iJ_{s-i}\right)~ \bigcap ~\left(\sum_{i=0}^s K_iJ_{s-i}\right)  \supseteq \sum_{i=0}^s (I_i\cap K_i) J_{s-i}.
	\end{equation}

	For the reverse containment, consider any element $x$ in the left hand side. Note that $x\in \sum_{i=0}^s I_iJ_{s-i} \subseteq I_s+J_1$. Similarly, $x\in K_s+J_1$. Thus, Lemma \ref{lem_intersection_sum} yields
	\[
	x\in (I_s+J_1)\cap (K_s+J_1)=(I_s\cap K_s)+J_1.
	\]
Write $x=x_1+x_2$, where $x_1 \in I_s \cap K_s$ and $x_2 \in J_1$. Then $x-x_1 =x_2\in J_1$. Since $x_1 \in I_s$, we have
		\[
		x-x_1\in \sum_{i=0}^s I_iJ_{s-i} +I_s = \sum_{i=0}^s I_iJ_{s-i} \subseteq \sum_{i=0}^{s-2} I_iJ_{s-i}+I_{s-1}.
		\]
Similarly, $x-x_1\in \sum_{i=0}^{s-2} K_iJ_{s-i}+K_{s-1}$.

Set $L_i=I_i\cap K_i$ for $i \geq 0$. Take $M_0=J_0$ and $M_i=J_{i+1}$ for all $i \geq 1$. Observe that $\{M_j\}_{j \ge 0}$ is a filtration of ideals of $B$. Also, note that \[\sum_{i=0}^{s-2} I_iJ_{s-i}+I_{s-1}=\sum_{i=0}^{s-1} I_iM_{s-1-i} \text{ and }  \sum_{i=0}^{s-2} K_iJ_{s-i}+K_{s-1} =\sum_{i=0}^{s-1} K_iM_{s-1-i}.\] Thus, by the induction hypothesis,
		\[
		\left(\sum_{i=0}^{s-1} I_iM_{s-1-i} \right) ~ \bigcap ~ \left(\sum_{i=0}^{s-1} K_iM_{s-1-i} \right) =\sum_{i=0}^{s-1} L_iM_{s-1-i} =\sum_{i=0}^{s-2}L_iJ_{s-i} +L_{s-1}.
		\]
Furthermore, thanks to Lemma \ref{lem_intersection_binomsum_1}, we have
		\[
		J_1 ~ \bigcap ~ \left(\sum_{i=0}^{s-2} L_iJ_{s-i}+L_{s-1}\right)=\sum_{i=0}^{s-1} L_iJ_{s-i}.
		\]
Therefore,  $x-x_1 \in \sum_{i=0}^{s-1} L_iJ_{s-i}$. It follows that $x \in \sum_{i=0}^{s} L_iJ_{s-i}$, as $x_1 \in L_s$.
		In other words, $x$ is contained in the right hand side of \eqref{eq_intersection}. This completes the induction and the proof of the lemma.
\end{proof}

\begin{lemma}\label{colon}
Let $\{I_i\}_{i \ge 0}$ and $\{J_j\}_{j \ge 0}$ be filtrations of ideals in $A$ and $B$, respectively, with $I_0 = A$ and $J_0 = B$. Let $a\in A$ be an element and let $K \subseteq A$ be an ideal. Then,
	\begin{enumerate}[\quad \rm (1)]
		\item $\displaystyle  \left(\sum_{i=t}^s I_i J_{s-i} \right):_R a \subseteq  (I_t :_A a)J_{s-t}+\left(\sum_{i=t+1}^s  I_i J_{s-i}\right):_R a$ for any $0 \leq t \leq s-1$.
		\item $\displaystyle  \left(\sum_{i=0}^s I_i J_{s-i} \right):_R K =  \sum_{i=0}^s \left( I_i :_A K\right) J_{s-i}.$
	\end{enumerate}
\end{lemma}

\begin{proof} (1) Let $x \in \left(\sum_{i=t}^s I_i J_{s-i} \right):_R a$ be any element. Then, $xa \in \left(\sum_{i=t}^s I_i J_{s-i} \right) \subseteq I_t$ since $\{I_i\}_{i \ge 0}$ is a filtration of ideals in $R$. Consequently, $x \in I_t :_R a= (I_t:_A a)R$.
	
	Also, $xa \in I_tJ_{s-t}+\left(\sum_{i= t+1}^s I_i J_{s-i} \right)$. Thus, there exists $y \in \sum_{i=t+1}^s I_i J_{s-i} $ such that $xa-y \in I_tJ_{s-t} \subseteq J_{s-t}R.$ Note that $\{I_i\}_{i \ge 0}$ is a filtration, and so $y \in I_{t+1}R$.  Therefore, $xa-y \in \left((I_t:_A a) a+ I_{t+1}\right)R$, which implies that
	\begin{align*}
		xa-y & \in \left( (I_t:_A a) a+ I_{t+1}\right)R \cap J_{s-t}R \\
		&= \left( (I_t:_A a) a+ I_{t+1}\right) J_{s-t} \quad \text{(by Lemma \ref{lem_intersection})}\\
		&=(I_t:_A a) aJ_{s-t}+ I_{t+1}J_{s-t}.
	\end{align*}

	Next, choose $z \in (I_t:_A a)J_{s-t}$ such that $xa-y -za \in I_{t+1}J_{s-t} \subseteq I_{t+1}J_{s-t-1}$, where the last inclusion follows from the fact that $\{J_j\}_{j \ge 0}$ is a filtration. In particular, 
$$(x-z)a \in (y) + I_{t+1}J_{s-t-1} \subseteq  \sum_{i=t+1}^s I_i J_{s-i}.$$ It follows that  $x \in (I_t:_A a)J_{s-t} + \left( \sum_{i=t+1}^s I_i J_{s-i}\right):_R a$, as desired.
	
(2) Consider first the case where $K=(a)$ is a principal ideal.  By a recursive use of part (1), we get
	$$ \left(\sum_{i=t}^s I_i J_{s-i} \right):_R a \subseteq \sum_{i=t}^s \left( I_i :_A a\right) J_{s-i} \text{ for } 0 \leq t \leq s-1.$$
	The reverse containment follows from the fact that, for every $ 0 \leq i \leq s$,
	$$( I_i :_A a)J_{s-i} \subseteq (I_iJ_{s-i}):_R a \subseteq \left(\sum_{i=0}^s I_i J_{s-i} \right):_R a.$$
	Hence, the assertion holds when $K$ is a principal ideal.
	
	Consider the general case where $K$ is an arbitrary ideal. Since $A$ is Noetherian, $K$ is finitely generated, i.e., $K=(a_1,\ldots,a_d)$ for $a_1, \dots, a_d \in A$. Applying Lemma \ref{lem_intersection_binomsum_long} and making use of the case where $K$ is principal, we obtain
	\begin{align*}
		\left(\sum_{i=0}^s I_i J_{s-i} \right):_R K  &=\bigcap_{t=1}^d \left(\sum_{i=0}^s I_i J_{s-i} :_R a_t \right)
		 =\bigcap_{t=1}^d \left(\sum_{i=0}^s (I_i :_A a_t) J_{s-i} \right)\\
		&= \sum_{i=0}^s \left(\bigcap_{t=1}^d (I_i :_A a_t)\right) J_{s-i}
		 = \sum_{i=0}^s \left( I_i :_A K\right) J_{s-i}.
	\end{align*}
	The conclusion follows.
\end{proof}

Our main result, the binomial expansion for saturated powers of $(I+J)$, is completed with the final step in the following statement.

\begin{theorem}
\label{binomial_expansion}
	Let $ I, K \subseteq A$, and $J, L \subseteq B$ be ideals. For any $s \in \NN$, we have
	$$ \displaystyle (I+J)^{(s)}_{KL}=\sum_{i=0}^s I^{(i)}_K J^{(s-i)}_L .$$
\end{theorem}

\begin{proof}
	Since $R$ is Noetherian, there exists a positive integer $t$ such that $(I+J)^{(s)}_{KL}= (I+J)^s:_R (KL)^{t}$. By applying Lemma \ref{colon}, we then obtain the following inclusion of ideals in $R$:
	\begin{align*}
		(I+J)^s:_R (KL)^{t} & =  (I+J)^s:_R K^tL^t 
		= \left((I+J)^s:_R K^t\right):_RL^t \\
		& = \left(\sum_{i=0}^s \left( I^i :_A K^t\right) J^{s-i} \right):_R L^t
		 = \sum_{i=0}^s \left( I^i :_A K^t\right) \left(J^{s-i} :_B L^t\right) \\
		& \subseteq \sum_{i=0}^s I^{(i)}_K J^{(s-i)}_L.
	\end{align*}
This, together with the inclusion in Proposition \ref{prop_sub}, establishes the desired equality.
\end{proof}

As an immediate consequence of Theorem \ref{binomial_expansion}, we obtain the following corollary.

\begin{corollary} \label{cor.sat=power}
	Let $ I, K$ be ideals of $A$, and let $J,L$ be ideals of $B$. If $I^{(i)}_K =I^i$ and $J^{(i)}_L =J^i$ for all $i \leq s$, then $ (I+J)^{(s)}_{KL}=(I+J)^s$.
\end{corollary}

\begin{proof} The assertion is straightforward from the binomial expansion of $(I+J)^{(s)}_{KL}$ given in Theorem \ref{binomial_expansion}.
\end{proof}

The following result says that the converse of Corollary \ref{cor.sat=power} also holds in the most relevant cases, for example, when each of $A$ and $B$ is either a domain or a standard graded $\kk$-algebra,  and $I$ and $J$ are proper, non-nilpotent ideals. This result is a generalization of \cite[Corollary 3.5]{HNTT}.

\begin{theorem} \label{thm.equalityConverse}
	Let $I, K \subseteq A$ and $J, L \subseteq B$ be ideals, and let $s\in \NN$. Assume that $I^{s-1}\neq I^s$ and $J^{s-1}\neq J^s$. Then, $ (I+J)^{(s)}_{KL} =(I+J)^s$ implies that $I^{(i)}_K =I^i$ and $J^{(i)}_L =J^i$ for all $i \leq s$.
\end{theorem}

\begin{proof}
Since $I^{s-1}\neq I^s$ and $J^{s-1}\neq J^s$, $I\subseteq A$ and $J\subseteq B$ are nonzero proper ideals.	Since $(I+J)^{(s)}_{KL}=(I+J)^s$, we have
	$\Ass_R((I+J)^{(s)}_{KL}) = \Ass_R((I+J)^s).$
	Therefore, $KL \not\subseteq P$ for any $P \in \Ass_R((I+J)^s)$. This implies that $K, L  \not\subseteq P$ for any $P \in \Ass_R((I+J)^s)$.

	Fix $1 \leq i \leq s$. Since $I^{s-1}\neq I^s$, it follows that $I^{i-1}\neq I^i$, and hence $\Ass_A(I^{i-1}/I^i) \neq \emptyset$. Similarly $\Ass_B(J^{j-1}/J^j) \neq \emptyset$ for any $1\le j\le s$. Consider any $\pp \in \Ass_A(I^{i-1}/I^i)$. Observe that if $K \subseteq \pp$ then, for any $\qq \in \Ass_B(J^{s-i}/J^{s-i+1})$, we have $K \subseteq P$ for all $P \in \Min_R(\pp+\qq)$. Moreover, by Lemma \ref{lem_ass_powers}, we have
	\[
	\bigcup_{i=1}^s \mathop{\bigcup_{\pp \in \Ass_A\left(I^{i-1}/I^i\right)}}_{\qq \in \Ass_B\left(J^{s-i}/J^{s-i+1}\right)} \Min_R(\pp+\qq) \subseteq \Ass_R((I+J)^s).
	\]
	Thus, any such $P$ is an element of $\Ass_R((I+J)^s)$, and this yields a contradiction. Therefore, $K \not\subseteq \pp$ for every $\pp \in \Ass_A(I^{i-1}/I^i)$. Since
	$$
	\Ass_A(A/I^i)\subseteq \bigcup_{j=1}^i\Ass_A(I^{j-1}/I^j),
	$$
	it follows that $K\not\subseteq \pp$ for any $\pp \in \Ass_A(A/I^i)$. We deduce by definition that $I^{(i)}_K=I^i$. Similarly, we get $J^{(i)}_L=J^i$ for all $1 \leq i \leq s$. The theorem is proved.
\end{proof}

Theorem \ref{binomial_expansion} gives a formula for the saturated powers of $(I+J)$ with respect to a product $KL$, where $K \subseteq A$ and $L \subseteq B$. The following question arises naturally: Let $ I$ and $J$ be ideals in $A$ and $B$ respectively. Let $E$ be an ideal in $R$. Do there exist ideals $K \subseteq A$ and $L \subseteq B$ such that
	$$ \displaystyle (I+J)^{(s)}_E=\sum_{i=0}^s I^{(i)}_K J^{(s-i)}_L ?$$
\noindent This question has a negative answer in general, as the next example demonstrates.

\begin{example}
Take $A = \kk[x, y], I = (x^2, xy), B = \kk[z, t], J = (z^2 , zt)$. Take $E = (x, y, z, t)$ in
$R = A \otimes_{\kk} B = \kk[x, y, z, t]$. Then $(I + J) :_R E^\infty = (xz, x^2, xy, z^2, zt)$. This is because $I+J$ admits the following primary decomposition
\[
I + J = (x, z) \cap (x, z^2 , t) \cap (x^2 , y, z) \cap (x^2 , y, z^2 , t),
\]
and $E = (x, y, z, t)$ is contained in only the radical of the last of these primary ideals. Hence,
\[
(I + J) :_R E^\infty = (x, z) \cap (x, z^2, t) \cap (x^2 , y, z) = (xz, x^2 , xy, z^2 , zt).
\]

On the other hand, assume that there exist ideals $K$ and $L$ in $A$ and $B$, respectively, such
that $(I + J) :_R E^\infty = (I :_A K^\infty) + (J :_B L^\infty)$. Observe that $I = (x) \cap (x^2 , y)$, so $I :_A K^\infty$ can
only be either $(x)$, $(x^2 , xy)$ or $(1)$.

Similarly $J :_B L^\infty$ can only be either $(z)$, $(z^2, zt)$ or $(1)$. In any case, we cannot
have $(I :_A K^\infty) + (J :_B L^\infty)$ being equal to $(xz, x^2 , xy, z^2 , zt)$.
\end{example}

%%%%%%%%%%%%%%%%%%%%%%%%%%%%%%%%%

\section{Binomial expansion for symbolic powers} \label{sec.applications}

In this section, we apply Theorem \ref{binomial_expansion} to obtain the binomial expansion (\ref{eq.bin}) for both known notions of symbolic powers. We further derive formulas for the depth and regularity of saturated powers of sums of ideals, generalizing those given in \cite{HNTT} for symbolic powers. Throughout this section, $A$ and $B$ shall denote Noetherian $\kk$-algebras such that $R = A \otimes_\kk B$ is also Noetherian.

The following binomial expansion was given in \cite{HNTT} for symbolic powers defined using minimal primes. We provide an alternative proof that, at the same time, works also for symbolic powers defined using all associated primes.

\begin{theorem}
	\label{thm.symbBIN}
	Let $I \subseteq A$ and $J \subseteq B$ be nonzero proper ideals. Then, for any $s \in \NN$, we have
	\begin{align} \label{eq.symbBIN}
		(I+J)^{(s)} = \sum_{i=0}^s I^{(i)}J^{(s-i)}.
	\end{align}
	Here, the equality is valid for both known definitions of symbolic powers.
\end{theorem}

\begin{proof}
We shall first establish the equality (\ref{eq.symbBIN}) for symbolic powers that are defined using the associated primes, employing Lemma \ref{lem_satpower_ass_symbpower}.

Set $$K = \bigcap_{\substack{\pp \in \Ass_A^*(I) \\ \grade(\pp,A/I) \ge 1}} \pp \text{ and } L = \bigcap_{\substack{\qq \in \Ass_B^*(J) \\ \grade(\qq,B/J) \ge 1}} \qq.$$

By applying Lemma \ref{lem_satpower_ass_symbpower}, Theorem \ref{binomial_expansion}, and Lemma \ref{lem.symbSat_ass}, we  obtain
	$$\up{a}(I+J)^{(s)}  = (I+J)^{(s)}_{KL} = \sum_{i=0}^s I^{(i)}_{K}J^{(s-i)}_{L} = \sum_{i=0}^s \ \up{a}I^{(i)} \ \up{a}J^{(s-i)}.$$
The proof for symbolic powers defined in terms of minimal primes is similar, employing Lemma \ref{lem_satpower_min_symbpower},  Theorem \ref{binomial_expansion}, and Lemma \ref{lem.symbSat_min}. This completes the proof of the theorem.
\end{proof}

As a consequence of Theorem \ref{thm.symbBIN}, we obtain an analog of \cite[Corollary 3.5]{HNTT} for symbolic powers defined using all associated primes.

\begin{corollary}\label{cor.equality}
	Let $I \subseteq A$ and $J \subseteq B$ be ideals, and let $s\ge 1$ be an integer. Assume that $I^{s-1}\neq I^s$ and $J^{s-1}\neq J^s$. If $\up{a}(I+J)^{(s)} =(I+J)^s$ then $\up{a}I^{(i)} =I^i$ and $\up{a}J^{(i)} =J^i$ for all $1\le i \leq s$.
\end{corollary}

\begin{proof} 
The hypotheses imply that $I\subseteq A, J\subseteq B$ are nonzero proper ideals. The assertion follows from Theorems \ref{thm.equalityConverse} and \ref{thm.symbBIN}, and Lemma \ref{lem_satpower_ass_symbpower}.
\end{proof}

The condition that $I^{s-1} \neq I^s$ and $J^{s-1} \neq J^s$ in Corollary \ref{cor.equality} is necessary, as illustrated in the following example.
\begin{example}
Let $A=\kk[x]/(x-x^2)$ and $B=\kk[u,v,w,z]$. Take $I=(x) \subseteq A$ and $J=(u^5,u^4v,uv^4,v^5,u^3v^3,u^3v^2w+u^2v^3z)\subseteq B.$ Note that $\up{a}I^{(n)}=I^n=I^{n+1} =(x)$ for all $n \in \mathbb{N}.$ 
Also, $J$ is a $(u,v)$-primary ideal of $B$. By Example \ref{ex.symbAsat}, $$J^2=(u^{10},u^9v,u^8v^2,u^6v^4,u^5v^5,u^4v^6,u^2v^8,uv^9,v^{10},u^7v^3w,u^3v^7w,u^7v^3z,u^3v^7z),$$ 
and $J^s=(u,v)^{5s}$ for all $s \geq 3$. Therefore, $\Ass_B^*(J)=\{(u,v),(u,v,w,z)\}$. 

Take $L=(u,v,w,z)$. Then,  $w\in L$ is a regular element over $B/J$. Therefore, by Lemma \ref{lem.symbSat_ass},  $\up{a}J^{(2)} \neq J^2$ and $\up{a}J^{(3)}=J^3$. However, $\up{a}(I+J)^{(3)}=(I+J)^3 = (x) + J^3$, by Theorem \ref{thm.symbBIN}.
\end{example}

We end the paper by using Theorem \ref{binomial_expansion} to derive formulas for the depth and regularity of saturated powers of $(I+J)$, generalizing those given in \cite{HNTT} for symbolic powers $\up{m}(I+J)^{(s)}$ (and, thus, exhibiting that those formulas hold for $\up{a}(I+J)^{(s)}$ as well). For simplicity of notations, set $[1,s] = \{1, \dots, s\}$. 

For the remaining results, let $A=\kk[x_1,\ldots,x_d]$ and $B=\kk[y_1,\ldots,y_e]$ be standard graded polynomial rings over $\kk$. Let $I \subseteq A$ and $J \subseteq B$ be nonzero proper homogeneous ideals. 

\begin{theorem}
\label{thm_depth_reg_saturatedpowers}
Let $K\subseteq A$ and $L\subseteq B$ be homogeneous ideals. Assume that either of the following conditions holds:
\begin{enumerate}[\quad \rm (i)]
\item $\chara \kk=0$;
\item $I, K, J,L$ are monomial ideals.
\end{enumerate}
Then, for all $s\ge 1$, there are equality
\begin{align*}
\depth \dfrac{R}{(I+J)^{(s)}_{KL}} & =\min_{i\in [1,s]} \left\{\depth \dfrac{A}{I^{(i)}_K}+\depth \dfrac{B}{J^{(s-i)}_L}+1, \depth \dfrac{A}{I^{(i)}_K}+\depth \dfrac{B}{J^{(s+1-i)}_L}\right\},\\
\reg \dfrac{R}{(I+J)^{(s)}_{KL}} & =\max_{i\in [1,s]} \left\{\reg \dfrac{A}{I^{(i)}_K}+\reg \dfrac{B}{J^{(s-i)}_L}+1, \reg \dfrac{A}{I^{(i)}_K}+\reg \dfrac{B}{J^{(s+1-i)}_L}\right\}.
\end{align*}
\end{theorem}

The key step in the proof of Theorem \ref{thm_depth_reg_saturatedpowers} is given by the following lemma.

\begin{lemma}
\label{lem_Torvanishing} 
With the above notations, assume that either of the following conditions holds:
\begin{enumerate}[\quad \rm (i)]
\item $\chara \kk=0$;
\item $I, K$ are monomial ideals.
\end{enumerate}
If $\chara \kk = 0$ then denote by $\partial(I)=\left(\partial(f)/\partial x_i ~\big|~ f\in I, i=1,\ldots,d\right)$ the ideal generated by partial derivatives of elements in $I$. If $I$ is a monomial ideal then set
$$
\partial^*(I)=\left(f/x_i ~\big|~ f ~~ \text{is a monomial in $I$}, x_i ~~ \text{divides}~~ f\right).
$$ 
Let $s\ge 1$ be an integer. Then the following statements hold.
\begin{enumerate}[\quad \rm (1)]
\item If $\chara\kk=0$ then $\partial(I^{(s)}_K)\subseteq I^{(s-1)}_K$. If $I$ and $K$ are monomial ideals, then so is $I^{(s)}_K$ and $\partial^*(I^{(s)}_K)\subseteq I^{(s-1)}_K$.

\item In either case, for all $s\ge 1$ and all $i\in \ZZ$, the map $\Tor^A_i(\kk, I^{(s)}_K) \to \Tor^A_i(\kk, I^{(s-1)}_K)$ is zero.
\end{enumerate}
\end{lemma}

\begin{proof}
(1) If $\chara \kk=0$ then arguing as in the proof of \cite[Proposition 5.5]{HNTT}, where the key point is the product rule for derivatives, we get $\partial(I^{(s)}_K)\subseteq I^{(s-1)}_K$. 

Assume that $I$ and $K$ are monomial ideals. Clearly $I^{(s)}_K=I^s:K^\infty$ is a monomial ideal. Take a monomial  $f\in \partial^*(I^{(s)}_K)$, we show that $f\in I^{(s-1)}_K$. By the definition, $fx_i \in I^{(s)}_K$ for some $1\le i\le d$. Thus for some $m\gg 0$, we have $fx_iK^m\subseteq I^s$. As $f$ is a monomial and $I,K$ are monomial ideals, this implies
\[
fK^m \subseteq \partial^*(I^s)\subseteq I^{s-1}.
\]
Hence $f\in I^{(s-1)}_K$, as claimed.

(2) The conclusion follows from the differential criteria for the vanishing of Tor maps given in \cite[Proposition 3.5]{AM} and \cite[Proposition 4.4 and Lemma 4.2]{NgV}, and recorded in \cite[Lemmas 5.4 and 5.9]{HNTT}.
\end{proof}

We are now ready to prove Theorem \ref{thm_depth_reg_saturatedpowers}.

\begin{proof}[Proof of Theorem \ref{thm_depth_reg_saturatedpowers}]
The proof follows closely to the arguments for \cite[Theorems 5.6 and 5.11]{HNTT}. Specifically, we consider the filtrations $\{I^{(i)}_K\}_{i\ge 0}$ in $A$ and $\{J^{(j)}_L\}_{j\ge 0}$ in $B$. By Lemma \ref{lem_Torvanishing}, these filtrations are Tor-vanishing in the sense that  $\Tor^A_i(\kk, I^{(s)}_K) \to \Tor^A_i(\kk, I^{(s-1)}_K)$ and $\Tor^B_i(\kk, J^{(s)}_L) \to \Tor^B_i(\kk, J^{(s-1)}_L)$ are zero maps for all $i$ and all $s\ge 1$.

On the other hand, by Theorem \ref{binomial_expansion},
\[
(I+J)^{(s)}_{KL}= \sum_{i=0}^s I^{(i)}_K J^{(s-i)}_L.
\]
Hence by \cite[Theorem 5.3]{HNTT}, the desired equality follow. Note that the original formulation of \textit{ibid.} yields formulas with slightly different index sets
\begin{align*}
\depth \dfrac{R}{(I+J)^{(s)}_{KL}} & =\min_{\substack{i\in [1,s-1] \\ j\in [1,s]}} \left\{\depth \dfrac{A}{I^{(s-i)}_K}+\depth \dfrac{B}{J^{(i)}_L}+1, \depth \dfrac{A}{I^{(s+1-j)}_K}+\depth \dfrac{B}{J^{(j)}_L}\right\},\\
\reg \dfrac{R}{(I+J)^{(s)}_{KL}} & =\max_{\substack{i\in [1,s-1] \\ j\in [1,s]}} \left\{\reg \dfrac{A}{I^{(s-i)}_K}+\reg \dfrac{B}{J^{(i)}_L}+1, \reg \dfrac{A}{I^{(s+1-j)}_K}+\reg \dfrac{B}{J^{(j)}_L}\right\},
\end{align*}
whereas our formulas read
\begin{align*}
\depth \dfrac{R}{(I+J)^{(s)}_{KL}} & =\min_{i\in [1,s]} \left\{\depth \dfrac{A}{I^{(i)}_K}+\depth \dfrac{B}{J^{(s-i)}_L}+1, \depth \dfrac{A}{I^{(i)}_K}+\depth \dfrac{B}{J^{(s+1-i)}_L}\right\},\\
\reg \dfrac{R}{(I+J)^{(s)}_{KL}} & =\max_{i\in [1,s]} \left\{\reg \dfrac{A}{I^{(i)}_K}+\reg \dfrac{B}{J^{(s-i)}_L}+1, \reg \dfrac{A}{I^{(i)}_K}+\reg \dfrac{B}{J^{(s+1-i)}_L}\right\}.
\end{align*}
However, both of these formulations agree, thanks to the standard convention that $\depth 0=+\infty$ and $\reg 0=-\infty$. We leave the details to the interested reader.
\end{proof}

Theorem \ref{thm_depth_reg_saturatedpowers} gives the following analogs of \cite[Theorems 5.6 and 5.11]{HNTT} for symbolic powers defined using all  associated primes.

\begin{corollary}
\label{cor_depth_reg_ass-symbpower}
Keep the notations of Theorem \ref{thm_depth_reg_saturatedpowers}. Assume that either of the following conditions holds:
\begin{enumerate}[\quad \rm (i)]
\item $\chara \kk=0$;
\item $I, J$ are monomial ideals.
\end{enumerate}
Then, for all $s\ge 1$, there are equality
\begin{align*}
\depth \dfrac{R}{\up{a}(I+J)^{(s)}} & =\min_{i\in [1,s]} \left\{\depth \dfrac{A}{\up{a}I^{(i)}}+\depth \dfrac{B}{\up{a}J^{(s-i)}}+1, \depth \dfrac{A}{\up{a}I^{(i)}}+\depth \dfrac{B}{\up{a}J^{(s+1-i)}}\right\},\\
\reg \dfrac{R}{\up{a}(I+J)^{(s)}} & =\max_{i\in [1,s]} \left\{\reg \dfrac{A}{\up{a}I^{(i)}}+\reg \dfrac{B}{\up{a}J^{(s-i)}}+1, \reg \dfrac{A}{\up{a}I^{(i)}}+\reg \dfrac{B}{\up{a}J^{(s+1-i)}}\right\}.
\end{align*}
\end{corollary}

\begin{proof} Observe that, by Lemma \ref{lem.symbSat_ass} and setting $K=\mathop{\bigcap\limits_{\pp \in \Ass_A^*(I)}}\limits_{\grade(\pp, A/I) \ge 1}\pp$, we have
\[
\up{a}I^{(s)}=I^s:K^\infty \quad \text{for all $s\ge 1$}.
\]
In particular, if $I$ is homogeneous then so is $K$. If $I$ is a monomial ideal then so are all the associated primes of its powers and, hence, so is $K$.

The desired conclusions follows from combining Lemma \ref{lem_satpower_ass_symbpower} and Theorems \ref{binomial_expansion} and \ref{thm_depth_reg_saturatedpowers}. 
\end{proof}
%%%%%%%%%%%%%%%%%%%%%%%%%%%%%%%%

\end{document}